\documentclass[reqno]{amsart}
\usepackage[bb=dsserif]{mathalpha}
\usepackage{hyperref}
\usepackage{float}
\usepackage{amsmath}
\usepackage{graphicx} 
\usepackage{latexsym}
\usepackage{amsfonts}
\usepackage{amssymb}
\usepackage{color} 
\theoremstyle{plain}
\newtheorem{theorem}{Theorem}
\newtheorem{corollary}[theorem]{Corollary}
\newtheorem{lemma}[theorem]{Lemma}

\theoremstyle{definition}

\newtheorem{remark}[theorem]{Remark}
\newtheorem*{remark*}{Remark}

\newcommand{\pr}{\mathbb{P}}
\newcommand{\R}{\mathbb{R}}

\newcommand{\lf}{\lfloor}
\newcommand{\rf}{\rfloor}
\newcommand{\E}{\mathbb{E}}
\newcommand{\ph}{\varphi}
\newcommand{\Z}{\mathbb{Z}}

\title[Berry-Esseen for walks conditioned to stay positive]{Berry-Esseen inequality for random walks conditioned to stay positive}
\date{}
\thanks{Alexander Tarasov and Vitali Wachtel were supported by the DFG}

\author[Denisov]{Denis Denisov}
\address{Department of Mathematics, University of Manchester, UK}
\email{denis.denisov@manchester.ac.uk}

\author[Tarasov]{Alexander Tarasov}
\address{Faculty of Mathematics, Bielefeld University, Germany}
\email{atarasov@math.uni-bielefeld.de}

\author[Wachtel]{Vitali Wachtel}
\address{Faculty of Mathematics, Bielefeld University, Germany}
\email{wachtel@math.uni-bielefeld.de}

\keywords{Random walk, exit time, Rayleigh distribution, Berry-Esseen inequality}
    \subjclass{Primary 60G50; Secondary 60G40, 60F17}
\begin{document}

\begin{abstract}
    We consider random walks conditioned to stay positive. 
    When the mean of increments is zero and variance is finite 
    it is known that they converge to the Rayleigh distribution. In the present paper we derive a Berry-Esseen type estimate and show that the rate of convergence is of order $n^{-1/2}$.  
\end{abstract}

\maketitle
\section{Introduction}
Let  $\{X_k\}$ be a sequence of  independent random variables with zero mean 
$\E X_1=0$ and finite variance $\E X_1^2=:\sigma^2\in(0,\infty)$. 
Consider a random walk 
 $\{S_n; n\ge0\}$ defined as follows, 
 $S_0=0$ and 
\begin{align*}
    S_n:=X_1+X_2+\ldots+X_n,\ n\ge1.
\end{align*}
Under the assumption $\E |X_1|^3 < \infty$  the well-known Berry-Esseen inequality says  
\begin{equation}\label{eq.berry-esseen.classical}
    \left|\pr\left(\frac{S_n}{\sigma \sqrt{n}} \le x\right) - \Phi(x)\right|
\le
    \gamma_0\frac{\E |X_1|^3}{\sigma ^3\sqrt{n}},
\end{equation}
where $\Phi$ stands for the standard normal distribution function and  one can take $\gamma_0=0.4785$. 
Originally the inequality was proved in \cite{B41} and we refer to  \cite{T10} for the provided inequality with exact constant.
The Berry-Esseen's type of inequalities quantifies the error in the central limit theorem and has many related embellishments such as assuming independent, but not identically distributed summands, or allowing a specified dependence structure. 

In this paper we consider random walks conditioned to stay positive. A functional limit theorem showing convergence of these random walks to the Brownian meander was proved in~\cite{Iglehart74} under an extra assumption of the existence of the third moment of increments. An elegant and short proof which did not require this assumption  was later given in~\cite{Bolthausen76}. These results imply in particular that properly scaled conditioned random walk converges to the Rayleigh distribution, which is the marginal distribution of the Brownian meander at time $1$.

In the present paper we are concerned with the rate of convergence to the Rayleigh distribution. 
The main result of the paper is the Berry-Esseen inequality for random walks conditioned to stay positive. Namely, denoting by $\tau$ the first time when the walk $\{S_n\}$ is non-positive, that is,
\begin{align*}
    \tau:=\inf\{n\ge1:S_n\le0\},
\end{align*}
we prove the following theorem.
\begin{theorem}
\label{thm:BE}
Assume $\E X_1 = 0, \E |X_1|^2 = \sigma^2$ and $\E |X_1|^3 < \infty$. 
Then there exists an absolute constant $A$ such that
\begin{align}
\label{eq:BE1}
    \bigg|
        \pr\left(\frac{S_n}{\sigma\sqrt{n}} \ge x\mid  \tau>n\right) - e^{-x^2/2}
    \bigg|
\le
    A
    \frac{(\E |X_1|^3)^3}{\sigma^9\sqrt{n}}.
\end{align}
Furthermore,
\begin{align}
\label{eq:BE2}
    \bigg|
        \frac{\pr(\tau>n)}
        {\sqrt{\frac{2}{\pi}}
        |\E S_\tau|n^{-1/2}}
    -
        1
    \bigg|
\le
    A \frac{(\E |X_1|^3)^3}{\sigma^9 \sqrt{n}}.
\end{align}
\end{theorem}
This theorem improves a recent result by Grama and Xiao~\cite{GramaXiao24+} 
for conditioned random walks started at $0$. They have shown that if $\E|X|^{2+\delta}<\infty$ for some $\delta\in(0,1]$ then there exists $\rho=\rho(\delta)$ such that 
\begin{equation}
\label{eq:GX}
\bigg|
        \pr(S_n \ge \sigma x\sqrt{n}| \tau>n) - e^{-x^2/2}
    \bigg|=O(n^{-\rho}).
\end{equation}

In the present paper we assume that $\E|X_1|^3<\infty$ which corresponds to $\delta=1$ in Grama and Xiao~\cite{GramaXiao24+}. 
However, one can modify the arguments 
to show that under the assumption 
$\E||X_1|^{2+\delta}<\infty$ the 
rate of convergence in 
Theorem~\ref{thm:BE} will be $n^{-\delta/2}$ when $\delta\in (0,1].$ 
(In the case $\delta<1$ one gets actually the rate $o(n^{-\delta/2})$.
This follows from the corresponding estimates for the concentration function of $S_n$, see Theorem~III.6 in \cite{Petrov}.)  A further improvement, in comparison to \cite{GramaXiao24+}, is a possibility to compute the rate of convergence in \eqref{eq:BE1} and in \eqref{eq:BE2}. 
In the course of the proof we give a rather crude upper bound on the constant $A$; it was not our purpose to obtain a good numerical constant, we wanted only clarify the dependence of the rate on the distribution of summands.
As in the classical Berry-Esseen inequality \eqref{eq.berry-esseen.classical}, the rate of convergence depends on the distribution of increments via the Lyapunov ratio $\sigma^{-3}\E|X_1|^3$
only. But, in contrast to \eqref{eq.berry-esseen.classical}, we get the third power of the Lyapunov ratio. This is caused by our approach, and it is possible that this part is not optimal. 

It is also worth noting that under some further assumptions it is also possible to obtain asymptotic expansions for conditioned random walks. In this case, similarly to the unconditioned random walks, we need to assume existence of higher moments and  some assumptions about the structure of increments. 
This problem has been considered in~\cite{DTW23, DTW25a}.

Let $V(x)$ denote the renewal function corresponding to weak descending ladder epochs. It is well known that this function is harmonic for $\{S_n\}$ killed at leaving $(0,\infty)$, that is,
$$
V(x)=\E[V(x+X_1);x+X_1>0], \quad x>0.
$$
This property allows one to perform the Doob $h$-transform and to define a random walk conditioned to stay positive at {\it all} times. This process is a Markov chain with the following transition kernel:
$$
P^{(h)}(x,dy)=\frac{V(y)}{V(x)}\pr(x+X_1\in dy),\quad x,y>0.
$$
Let $\pr^{(h)}$ denote the corresponding probability measure.
Bryn-Jones and Doney~\cite{BJD06} proved a functional limit theorem for $S_n$ under $\pr^{(h)}$. Their result implies that 
$$
\lim_{n\to\infty}
\pr^{(h)}\left(\frac{S_n}{\sigma\sqrt{n}}\le x\right)
=\int_0^x\sqrt{\frac{2}{\pi}}y^2e^{-y^2/2}dy,\quad x>0.
$$
Using \eqref{eq:GX}, Hong and Sun~\cite{HongSun24} have shown that 
$$
\sup_{x\ge0}
\left|\pr^{(h)}\left(\frac{S_n}{\sigma\sqrt{n}}\le x\right)
-
\int_0^x\sqrt{\frac{2}{\pi}}y^2e^{-y^2/2}dy\right|
\le C\frac{\log n}{n^\rho}.
$$
Repeating their proof and using \eqref{eq:BE1} instead of \eqref{eq:GX}, one can get 
\begin{equation}
\label{eq:HS1}
\sup_{x\ge0}
\left|\pr^{(h)}\left(\frac{S_n}{\sigma\sqrt{n}}\le x\right)
-
\int_0^x\sqrt{\frac{2}{\pi}}y^2e^{-y^2/2}dy\right|
\le C\frac{(\E|X_1|^3)^3}{\sigma^9}\frac{\log n}{\sqrt{n}}
\end{equation}
provided that $\E|X_1|^3$ is finite. We believe that the logarithm can be removed in this inequality. To achieve that improvement, the optimal uniform rate obtained in Theorem~\ref{thm:BE} is not sufficient. One needs a non-uniform rate of convergence, which takes into account lower and large deviations. Some results in that direction with non-optimal exponents have been obtained in \cite{GramaXiao24+}. 

Typically, to obtain a rate of convergence in a limit theorem one derives first a rate of convergence for the corresponding characteristic functions and applies then an appropriate version of the inversion formula for Fourier transforms. We use a completely different approach based on a representation for $\pr(S_n\ge x,\tau>n)$ which  can be seen as a generalisation of the classical reflection principle for simple random walks.

Assume for a moment that $S_n$ is a simple random walk, that is, $\pr(X_1=1)=\pr(X_1=-1)=\frac{1}{2}$.  Then one has $S_{\tau} = - {\mathbb 1}_{\tau=1}$. This implies that for integer  $x\ge 1$, 
\begin{align*}
    \pr(S_n \ge x, \tau > n) 
=&
    \pr(S_n \ge x) 
-
    \pr(S_n \ge x, \tau < n)\\
=&
    \pr(S_n \ge x) 
-
    \pr(S_n \ge x, \tau = 1)
-
    \pr(S_n \ge x, 2 \le \tau < n)\\
=&
    \pr(S_n \ge x) 
-
    \frac{1}{2}\pr(S_{n-1} \ge x + 1)
-
    \sum_{k=2}^{n-1}
        \pr(\tau = k) \pr(S_{n-k} \ge x).
\end{align*}
We will now apply the reflection  principle. Similarly to the previous equation:
\begin{align*}
    0 
=
    \pr(S_n \le -x, \tau >n) 
=
    \pr(S_n \le -x) 
-&
    \frac{1}{2}\pr(S_{n-1} < -x+1)\\ 
-&
    \sum_{k=2}^{n-1} 
        \pr(\tau = k) \pr(S_{n-k} \le - x).
\end{align*}
Taking the difference and using symmetry we obtain
\begin{align*}
    \pr(S_n \ge x,\tau > n) \
=&
    \frac{1}{2} \pr(S_{n-1} \ge x-1)
-
    \frac{1}{2} \pr(S_{n-1} \ge x+1)\\
=&
    \frac{1}{2}
    \pr\big(S_{n-1} \in \{x-1, x\}\big).
\end{align*}
Plugging in $x=1$ we obtain also 
\[
\pr(\tau>n) = \frac{1}{2}
    \pr\big(S_{n-1} \in \{0, 1\}\big).
\]
Hence, 
\begin{align*}
&\pr(S_n \ge x\mid \tau > n)\\
&\hspace{1cm}
=\frac{\pr\big(S_{n-1} \in \{x-1, x\})}{\pr\big(S_{n-1} \in \{0, 1\}\big)}\\
&\hspace{1cm}
=\frac{1}{\pr(S_{n-1}=(n-1)\text{ mod 2})}
\begin{cases}
\pr(S_{n-1}=x-1),& n-x \text{ is even }\\ 
\pr(S_{n-1}=x)
& n-x \text{ is odd }
\end{cases}
.
\end{align*}
This equality connects the tail of the conditioned walk to the local probabilities of the original (unconditional) process and explains the appearance of the standard normal density $e^{-x^2/2}$ as the limiting tail for conditioned walks. Furthermore, this equality allows one to obtain the rate of convergence for conditioned walks from the corresponding results for local probabilities of unconditioned walks.

In the general case we also start with the equalities
\begin{align*}
\pr(S_n\ge x,\tau>n)&=\pr(S_n\ge x)-\pr(S_n\ge x,\tau<n),\\
0&=\pr(S_n\le -x)-\pr(S_n\le -x,\tau\le n).
\end{align*}
Taking the difference and using the Markov property, we get 
\begin{align*}
&\pr(S_n\ge x,\tau>n)\\
&\hspace{0.5cm}=\pr(S_n\ge x)-\pr(S_n\le -x)+\pr(S_n\le-x,\tau=n)\\
&\hspace{1cm}-\sum_{k=1}^{n-1}\int_0^\infty\pr(S_k\in-dy,\tau=k)
[\pr(S_{n-k}\ge x-y)-\pr(S_{n-k}\le-x+y)].
\end{align*}
This representation is the starting point in our approach. 

In the next section we first prove a simplified (numerical constants will be replaced by $O(1)$) version of Theorem~\ref{thm:BE} for symmetric lattice walks. Under these additional assumptions we can avoid some technical difficulties and give deeper insights into our method. It is worth mentioning that we do not use Fourier methods in this special case.  
Then we turn to the proof in the general case. Here we apply smoothing with an appropriately chosen random variable $U$. The main technical step in the general case consists in estimating
$$
\pr(S_n+U\ge x)-\pr(S_n+U\le -x)
-\pr(S_{n+1}+U\ge x)+\pr(S_{n+1}+U\le -x).
$$
We show that this function is uniformly bounded by $Cn^{-3/2}$. To do so we use the standard approach via Fourier transforms.
\section{Proof of the main theorem}
In this section we shall assume, without loss of generality, that 
$$
\sigma^2=1.
$$
\subsection{Proof in the case of symmetric lattice distributions.}
To illustrate our approach, we consider first the case when the distribution of $X_1$ is symmetric and lattice.
Furthermore, we assume that the distribution of $X_1$ is aperiodic, that is, $\mathbb{Z}$ is the minimal lattice for $X_1$. To simplify the calculations, we will not take any care of constants. So, all the constants appearing in this subsection may depend on the whole distribution of $X_1$.

We first notice that
\begin{align*}
    \pr(S_n \ge x, \tau >n ) 
=
    \pr( S_n \ge x) 
-
    \sum_{k=1}^{n-1}
    \sum_{y=0}^\infty
        \pr(S_k = -y, \tau = k) \pr(S_{n-k} \ge x+y).
\end{align*}
Using the symmetry of the distribution of $S_k$, we also have
\begin{align*}
    0
=&
    \pr( S_n \le -x) 
-
    \sum_{k=1}^{n-1}
    \sum_{y=0}^\infty
        \pr(S_k = -y, \tau = k) \pr(S_{n-k} \le y-x)
-
    \pr(S_n \le -x, \tau = n)\\
=&
    \pr( S_n \ge x) 
-
    \sum_{k=1}^{n-1}
    \sum_{y=0}^\infty
        \pr(S_k = -y, \tau = k) \pr(S_{n-k} \ge x-y)
-
    \pr(S_n \le -x, \tau = n).
\end{align*}
Taking the difference, we obtain
\begin{align}
\label{eq:symm-repr}
    \pr(S_n \ge x, \tau >n ) 
&=
    \sum_{k=1}^{n-1}
    \sum_{y=0}^\infty
        \pr(S_k = -y, \tau = k) \pr\big(S_{n-k} \in [x-y, x+y)\, \big)\\
&\hspace{1cm}+ \nonumber
    \pr(S_n \le -x, \tau = n).
\end{align}
Clearly,
\begin{align}
\label{eq:tau=n}
    \pr(S_n \le -x, \tau = n)
\le
    \pr(\tau = n) \le \frac{C}{n^{3/2}},
\end{align}
where the last inequality follows from Theorem~8 in \cite{WV09}.

To deal with the sum on the right-hand side of \eqref{eq:symm-repr} we need a good approximation for $\pr\big(S_{n-k} \in [x-y, x+y)\big)$. By the standard bound for concentration functions of sums of i.i.d. variables, 
\begin{align*}
    \pr\big(S_{n-k} \in [x-y, x+y)\big) \le \frac{Cy}{\sqrt{n-k}}.
\end{align*}

This implies that
\begin{multline} \label{eq:smallhalfsum}
    \sum_{k=\lf n/2 \rf+1}^{n-1} 
    \sum_{y=0}^\infty
        \pr(S_k = -y, \tau = k) \pr\big(S_{n-k} \in [x-y,x+y)\big)
\\
\le
    C
    \sum_{k=\lf n/2 \rf+1}^{n-1}
        \frac{1}{\sqrt{n-k}}
        \sum_{y=0}^\infty
            y\pr(S_k = -y, \tau = k).
\end{multline}
By Lemma 20 in \cite{WV09},
$$
\pr(S_{k-1}=z, \tau > k-1) \le C\frac{z}{k^{3/2}}.
$$
Hence
\begin{align*}
    \sum_{y=0}^\infty
        y\pr(S_k = -y, \tau = k)
=&
    \sum_{y=0}^\infty
        y
        \sum_{z=1}^{\infty}
            \pr(S_{k-1} = z, \tau > k-1) \pr(-X = z+y)\\
\le& \nonumber
    \frac{C}{k^{3/2}}
    \sum_{y=0}^{\infty}
    \sum_{z=1}^{\infty}
            yz \pr(-X -z = y).
\end{align*}
Notice further that
\begin{align*}
    \sum_{y=0}^{\infty}
    \sum_{z=1}^{\infty}
            yz \pr(-X = z+y)
&\le
    \frac 14\sum_{y=0}^{\infty}
    \sum_{z=1}^{\infty}
            (z+y)^2 \pr(-X =z+ y)\\
&\le\nonumber
    \frac 14 \sum_{u=0}^\infty
        u^3 \pr(-X = u)
\le 
    \E \big[ -X^3; X< 0 \big].
\end{align*}
Combining the last two inequalities, we have
\begin{align} \label{eq:expectationStau}
    \sum_{y=0}^\infty
        y\pr(S_k = -y, \tau = k)
=
    \E [-S_\tau; \tau=k] 
\le
    \frac{C}{k^{3/2}} \E \big[ -X^3; X< 0 \big].
\end{align}
Substituting \eqref{eq:expectationStau} into \eqref{eq:smallhalfsum}, one gets
\begin{align}
\label{eq:large_k}
\nonumber
    \sum_{k=\lf n/2 \rf+1}^{n-1} 
    \sum_{y=0}^\infty
        \pr(S_k = -y, \tau = k) &\pr\big(S_{n-k} \in [x-y,x+y)\big)
\\
&\le
    C
    \sum_{k=\lf n/2 \rf+1}^{n-1}
        \frac{1}{\sqrt{n-k}}
        \frac{1}{k^{3/2}}
\le
    \frac{C}{n}.
\end{align}
Thus, it remains to consider $k\le \lf n/2\rf$. 
According to Theorem 4 in \cite{GJP1984}, for every $z \in [x-y, x+y)$,
\begin{align*}
\big|
    \pr(S_{n-k} =z) - \pr(S_{n-k} = x)         
\big|
\le
    \frac{Cy}{n-k}.
\end{align*}
This implies that
\begin{align*}
\big|
    \pr\big( S_{n-k} \in [x-y, x+y) \big) - 2y \pr(S_{n-k} = x)
\big|
\le
    \frac{Cy^2}{n-k}
\end{align*}
and, consequently,
\begin{align}
\label{eq:small_k.1}
\nonumber
\Bigg|
    \sum_{k=1}^{\lf n/2 \rf}
    \sum_{y=1}^\infty
        \pr(S_k = -y, \tau = k)
        &
        \pr\big(S_{n-k} \in [x-y, x+y) \big)\\
        \nonumber
&-
    2 
    \sum_{k=1}^{\lf n/2 \rf}
        \pr(S_{n-k} = x)
        \sum_{y=1}^\infty
        y\pr(S_k = -y, \tau = k)
\Bigg|\\
\le 
    \frac{C}{n} 
    \sum_{k=1}^{\lf n/2 \rf}
    &
    \sum_{y=1}^\infty
        y^2 \pr( S_k = -y, \tau = k)
\le
    \frac{C}{n} \E [S^2_\tau] 
\le 
    \frac{C}{n}.
\end{align}
According to Theorem ~VII.6 in \cite{Petrov},
\begin{align*}
    \pr(S_{n-k} = x)
=
    \frac{1}{\sqrt{2\pi (n-k)}} e^{-\frac{x^2}{2(n-k)}}
+
    O\left(\frac{1}{n}\right)
\end{align*}
uniformly in $x$ and in $k\le n/2$. Using this estimate and the inequality
\begin{align*}
    \sum_{k=1}^{\lf n/2 \rf}
        \E [-S_\tau; \tau = k] \le \E [-S_\tau]<\infty,
\end{align*}
we conclude that
\begin{multline*}
    2 
    \sum_{k=1}^{\lf n/2 \rf}
        \pr(S_{n-k} = x)
        \E [-S_\tau; \tau = k]
=
    \sqrt{\frac{2}{\pi}}
    \sum_{k=1}^{\lf n/2 \rf}
        \frac{\E [-S_\tau; \tau = k]}{\sqrt{n-k}}
        e^{-\frac{x^2}{2(n-k)}}
+
    O
    \left(\frac{1}{n}\right).
\end{multline*}
Later we shall show (see \eqref{eq:cont_of_normal} below) that
\begin{align*}
\left|
    \frac{1}{\sqrt{n-k}} e^{-\frac{x^2}{2(n-k)}}
-  
    \frac{1}{\sqrt{n}} e^{-\frac{x^2}{2n}}
\right|
\le
    C\frac{k}{n^{3/2}}
\end{align*}
uniformly in $x$ and in $k \le n/2$.
Using \eqref{eq:expectationStau} we infer that
\begin{align*}
    \sum_{k=1}^{\lf n/2 \rf}
        \frac{k}{n^{3/2}}\E [-S_\tau;\tau = k]
\le
    \frac{C}{n^{3/2}} 
    \sum_{k=1}^{\lf n/2 \rf}
        k\frac{1}{k^{3/2}}
=
    O\left(
        \frac{1}{n}
    \right).
\end{align*}
Therefore,
\begin{align*}
    &\sqrt{\frac{2}{\pi}}
    \sum_{k=1}^{\lf n/2 \rf}
        \frac{\E [-S_\tau; \tau = k]}{\sqrt{n-k}}
        e^{-\frac{x^2}{2(n-k)}}\\
&\hspace{1cm}=
    \sqrt{\frac{2}{\pi}}
    \frac{1}{\sqrt{n}}
    e^{-\frac{x^2}{2n}}
    \sum_{k=1}^{\lf n/2 \rf}
        \E [-S_\tau; \tau = k] 
+
    O\left(
        \frac{1}{n}
    \right)\\
&\hspace{1cm}=
    \sqrt{\frac{2}{\pi}}
    \frac{1}{\sqrt{n}}
    \Big(\E [-S_\tau] - \E [-S_\tau; \tau \ge \lf n/2 \rf]\Big) e^{-\frac{x^2}{2n}}
+
    O\left(
        \frac{1}{n}
    \right)\\
&\hspace{1cm}=
    \sqrt{\frac{2}{\pi}}
    \frac{1}{\sqrt{n}}
    \E [-S_\tau] e^{-\frac{x^2}{2n}}
+
    O\left(
        \frac{1}{n}
    \right).
\end{align*}
Combining this with \eqref{eq:large_k} and \eqref{eq:small_k.1},
we arrive at the equality
\begin{align*}
    \pr(S_n \ge x, \tau > n) 
=
    \frac{2}{\pi} \E[-S_\tau]
    \frac{1}{\sqrt{n}}
    e^{-\frac{x^2}{2n}}
+
    O\left(
        \frac{1}{n}
    \right).
\end{align*}
Taking here $x=0$, we have
\begin{align*}
    \pr(\tau > n) 
=
    \frac{2}{\pi} \E[-S_\tau]
    \frac{1}{\sqrt{n}}
+
    O\left(
        \frac{1}{n}
    \right).
\end{align*}
Combining these two equalities, we conclude that for symmetric case
\begin{align*}
\bigg|
    \pr(S_n \ge x | \tau > n)
-
    e^{-\frac{x^2}{2n}}
\bigg|
\le
    \frac{C}{\sqrt{n}}
\end{align*}
for some constant $C$.

\subsection{Proof in the general case.}
The proof we have presented in the previous subsection can be repeated in the absolutely continuous symmetric case. If one does not assume that the distribution of $X_1$ is either lattice or absolute continuous then it is 
in general much harder to derive bounds for the rate of convergence. Usually one uses an appropriate smoothing procedure. For example, in the proof of the classical Berry-Esseen inequality one convolves $S_n$ with an independent of $\{S_n\}$ random variable $W$ which has the density 
$$
f_A(x)=\pi^{-1}\frac{1-\cos (Ax)}{Ax^2}, \quad x\in \R.
$$
It is clear that $\E W$ does not exist, which  makes $W$ inappropriate for our purposes. We shall use an independent smoothing random variable $U$ with the density 
$$
g_A(x)= \frac{3}{\pi A} \left(\frac{1-\cos (Ax)}{Ax^2}\right)^2, 
\quad x\in \R,
$$
where 
\[
A := \frac{1}{8 \E|X|^3}.
\] 
Noting that $g_A(x)=\frac{3\pi}{A}f_A^2(x)$ and using the duality between
densities and their Fourier transforms, we infer that the characteristic function
$\widehat{g}_A(t)$ of $U$ equals to a properly normed $2$-fold convolution 
of the characteristic function $\widehat{f}_A(t)$ of the random variable $W$, a closed form expression for this characteristic function can be found in Nagaev~\cite{Nagaev69}.

It is known that 
$$
\widehat{f}_A(t)=\left(1-\frac{|t|}{A}\right)\mathrm{1}\{t\in[-A,A]\}.
$$
This implies that $\widehat{g}_A(t)$ equals zero for all $t$ with $|t|>2A$.
It is also clear that $U$ is symmetric. In our later calculations we will need the first absolute moment of $U$, which can be calculated explicitly:
\[
\E|U|= \frac{1}{A} \frac{6\ln 2}{\pi}=\frac{48\ln 2}{\pi}\E|X_1|^3. 
\]

The first step in our proof is a comparison of probabilities 
$\pr(S_n\ge x,\tau>n)$ and $\pr(S_n+U\ge x,\tau>n)$. To this end we derive first an upper bound for $\pr(S_n\in(x,x+y],\tau>n)$. Using the classical Berry-Esseen inequality~\eqref{eq.berry-esseen.classical}, we conclude that, uniformly in $x$, 
\begin{align}
\nonumber 
  \pr(S_n\in [x,x+y]) 
  &\le 2\gamma_0\frac{\E|X_1|^3}{\sqrt n}
+\frac{1}{\sqrt{2\pi}}
\int_{x/\sqrt n}^{(x+y)/\sqrt n} e^{-z^2/2}dz 
\\\label{eq.conc}
&\le \frac{\gamma_1(y)}{\sqrt{2n}}, 
\end{align}
where $\gamma_1(y):=\sqrt 2 \E|X_1|^3+y \pi^{-1/2}$. 
This implies that 
\begin{align}
\label{sn.tau}
\nonumber
&\pr(S_n\in [x,x+y],\tau>n)\\
\nonumber
&\hspace{1cm}\le\int_0^\infty\pr(S_{\lfloor n/2\rfloor}\in dz,\tau>n/2)
\pr(S_{n-\lfloor n/2\rfloor}\in \textcolor{red}{[} x-z,x-z+y])\\
&\hspace{1cm}\le
\frac{\gamma_1(y)}{\sqrt{n}}\pr(\tau>n/2).
\end{align}
\begin{remark}
The appearance of the third moment on the right hand side of \eqref{eq.conc} is the only reason for the third power of the Lyapunov ratio on the right hand side of \eqref{eq:BE1}. If one manages to replace $\E|X_1|^3$ by an absolute constant then one obtains \eqref{eq:BE1} with 
$\frac{\E|X_1|^3}{\sigma^3}$ instead of its third power.
\hfill$\diamond$
\end{remark}

\begin{lemma}
\label{lem:smoothing}
For all $n\ge1$ and all $x\ge 0$ we have 
\begin{align}\label{eq:smooth.1}
\nonumber
&\left|\pr(S_n+U \ge x, \tau >n )  
-\pr(S_n\ge x,\tau>n)\right|\\
&\hspace{2cm}\le 
\frac{\gamma_1(\E|U|)}{\sqrt{n}}\pr(\tau>n/2)
\end{align}
and
\begin{equation}\label{eq:smooth.2}
\pr(S_n+U \le -x, \tau >n )  \le 
\frac{\gamma_1(\E|U|)}{\sqrt{n}}\pr(\tau>n/2). 
\end{equation}
\end{lemma}
\begin{proof}
First, using~\eqref{sn.tau},  
\begin{align*}
&\int_0^{\infty} \pr(U\in -dy) 
\left|\pr(S_n-y \ge x, \tau >n )  
-\pr(S_n\ge x,\tau>n)\right|\\
&\hspace{1cm}=
\int_0^{\infty} \pr(U\in -dy) 
\pr(S_n\in [x,x+y),\tau>n)\\
&\hspace{1cm}\le 
\int_0^{\infty} \pr(U\in -dy) 
\frac{\pi^{-1/2}y+\sqrt{2}\E|X_1|^3}{\sqrt{n}}\pr(\tau>n/2)\\
&\hspace{1cm}= 
\frac{\pi^{-1/2}\E U^-+\sqrt{2}\E|X_1|^3\pr(U<0)}{\sqrt{n}}\pr(\tau>n/2)
.
\end{align*}
Second, using~\eqref{sn.tau} once again,  
\begin{align*}
&\int_0^{\infty} \pr(U\in dy) 
\left|\pr(S_n+y \ge x, \tau >n )  
-\pr(S_n\ge x,\tau>n)\right|\\
&\hspace{1cm}=
\int_0^{x} \pr(U\in dy) 
\pr(S_n\in [x-y,x),\tau>n)
+\pr(U>x)
\pr(S_n\in (0,x),\tau>n)\\
&\hspace{1cm}\le 
\int_0^{x} \pr(U\in dy)
\frac{\pi^{-1/2}y+\sqrt{2}\E|X_1|^3}{\sqrt{n}}\pr(\tau>n/2)\\
&\hspace{2cm}+
\pr(U>x)\frac{\pi^{-1/2}x+\sqrt{2}\E|X_1|^3}{\sqrt{n}}\pr(\tau>n/2)\\
&\hspace{1cm}\le 
\frac{\pi^{-1/2}\E U^++\sqrt{2}\E|X_1|^3\pr(U>0)}{\sqrt{n}}\pr(\tau>n/2).
\end{align*}
Combining these two inequalities we obtain~\eqref{eq:smooth.1}. 

The second claim follows again from \eqref{sn.tau}:
\begin{align*}
\pr(S_n+U \le -x, \tau >n)
&=\int_x^\infty \pr(U\in -dy)
\pr(S_n\in(0, y-x], \tau >n)\\
&\le 
\int_x^\infty \pr(U\in -dy)\frac{\pi^{-1/2}(y-x)+\sqrt{2}\E|X_1|^3}{\sqrt{n}}\pr(\tau>n/2)\\
&\le \frac{\pi^{-1/2}\E U^-+\sqrt{2}\E|X_1|^3\pr(U<0)}{\sqrt{n}}\pr(\tau>n/2). 
\end{align*}
Thus, the proof of the lemma is complete.
\end{proof}
\begin{remark}
Since we know the exact formula for $\E|U|$, we have 
$$
\gamma_1(\E|U|)=\left(\frac{48\log2}{\pi^{3/2}}+\sqrt{2}\right)
\E|X_1|^3.
$$
This absolute constant in front of $\E|X_1|^3$ is smaller than $7.39$.\hfill$\diamond$
\end{remark}
By the total probability law,
\begin{align*}
    &\pr(S_n+U \ge x, \tau >n ) \\
&\hspace{1cm}=
    \pr( S_n+U \ge x) 
-\pr( S_n+U \ge x,\tau\le n)\\
&\hspace{1cm}=
    \pr( S_n+U \ge x) 
-
    \sum_{k=1}^{n}
    \int_{0}^\infty
        \pr(\tau = k, S_k \in -dy) \pr(S_{n-k}+U \ge x+y)\\
&\hspace{1cm}=
    \pr( S_n+U \ge x) 
-
    \sum_{k=1}^{n}
        \pr(\tau = k) \pr(S_{n-k}+U\ge x)\\
&\hspace{2cm}+
    \sum_{k=1}^{n}
    \int_{0}^\infty
        \pr(\tau = k, S_k \in -dy) 
        \pr\big(S_{n-k}+U\in [x, x+y) \big)
\end{align*}
and
\begin{align*}
&\pr(S_n+U \le -x,\tau>n)\\
&\hspace{1cm}=
    \pr(S_n+U \le -x) 
-
    \pr(S_n+U \le -x, \tau \le n)\\
&\hspace{1cm}=
    \pr(S_n+U\le -x) 
-
        \sum_{k=1}^{n}
    \int_{0}^\infty
        \pr(\tau = k, S_k \in -dy) \pr(S_{n-k}+U \le y-x)\\
&\hspace{1cm}=
    \pr(S_n+U \le -x) 
-
    \sum_{k=1}^{n}
        \pr(\tau = k) \pr(S_{n-k}+U\le -x)\\
&\hspace{2cm}
-
    \sum_{k=1}^{n}
    \int_{0}^\infty
        \pr(\tau = k,S_k \in -dy ) 
                \pr(S_{n-k}+U \in(-x,-x + y])
       .
\end{align*}
Set
\begin{align}
\label{eq:Pn-def}
    P_n(x)   
&:=\pr(S_n+U\ge x,\tau>n)-\pr(S_n+U\le-x,\tau>n)\\
\nonumber
&=
    \pr(S_n+U \ge x)
-
    \pr(S_n+U \le -x) \\
\nonumber
&\hspace{0.5cm}-
    \sum_{k=1}^{n}
        \pr(\tau = k) 
        \big[ 
            \pr(S_{n-k}+U \ge x) - \pr(S_{n-k}+U \le -x )
        \big]\\
\nonumber        
&\hspace{0.5cm}+
    \sum_{k=1}^{n}
    \int_{0}^\infty
        \pr(\tau = k, S_k\in -dy)
        \pr\left( S_{n-k} +U \in (-x, -x+y] \cup [x, x+y) \right).
\end{align}
It is immediate from Lemma~\ref{lem:smoothing} that 
\begin{equation}
\label{eq:smooth}
\sup_{x\ge0}\bigl|\pr(S_n\ge x,\tau>n)-P_n(x)\bigr|\le
2\frac{\gamma_1(\E|U|)}{\sqrt{n}}\pr(\tau>n/2).
\end{equation}
Since the tail of $\tau$ decays as $n^{-1/2}$, the right hand side 
has the right order. For the tail of $\tau$ we can also obtain an upper bound with explicit constants.
To this end we note that, due to Lemma 25 in~\cite{DSW18},
$$
\pr(\tau>n) \le \frac{\E[S_n;\tau>n]}{\E[S_n^+]}.
$$
Applying the optional stopping theorem to the martingale $S_n$ with the stopping time $\tau\wedge n$, we infer that 
$$
\E[S_n;\tau>n]=-\E[S_\tau;\tau\le n]\le |\E[S_\tau]|.
$$
Consequently,
\begin{equation}\label{eq.tau}
\pr(\tau>n) \le \frac{|\E[S_\tau]|}{\E[S_n^+]}.
\end{equation}
We can estimate the denominator by using once again inequality \eqref{eq.berry-esseen.classical}:
\begin{multline*}
\E[S_n^+]=\int_0^\infty 
\pr(S_n>x) dx 
\ge 
\sqrt n\int_0^{1} 
\pr(S_n>y\sqrt n) dy\\
\ge 
\sqrt{n}
\int_0^{1} \overline \Phi(y)dy 
-0.4785 \E|X_1|^3
\ge 0.341\sqrt{n} -0.4785 \E|X_1|^3. 
\end{multline*}
Therefore
\begin{equation}\label{eq.tau1}
\pr(\tau>n) \le \frac{|\E[S_\tau]|}{0.341\sqrt{n}  -0.4785 \E|X_1|^3}.
\end{equation}
This implies that 
\begin{equation}
\label{eq.tau2}
\pr(\tau>n) \le \frac{6|\E[S_\tau]|}{\sqrt{n}}
\end{equation}
for all $n\ge 7.33(\E|X_1|^3)^2$. 

 We will make use of the following bounds.
\begin{lemma}
\label{lem:lorden}
There exists an absolute constant $C$ such that  
\begin{equation}\label{estau}
\E|S_\tau|\le \frac{\E|X_1|^3}{\E X_1^2}
\end{equation}
and 
\begin{equation}\label{estau2}
\E S^2_\tau
\le C\E|S_\tau|\frac{\E|X_1|^3}{\E X_1^2}
\le C\left(\frac{\E|X_1|^3}{\E X_1^2}\right)^2.
\end{equation}
\end{lemma}
\begin{proof}
Mogulskii~\cite{Mogulskii74} proved~\eqref{estau}. 
This paper also has also a bound for 
the second moment of $S_\tau$ in terms of the fourth moment $\E S^2_\tau\le \frac{2}{3}\frac{\E X_1^4}{\E X_1^2}$.

Bound~\eqref{estau2} follows from Theorem 5 in~\cite{Aleshkyavichene06} combined 
with Lemma~1 in~\cite{Aleshkyavichene73}. 
\end{proof}
This lemma can be used to bound the nominator in~\eqref{eq.tau1}.
Thus, we are left to analyse the function $P_n(x)$.

It was shown in \cite{WV09}, see Lemma 20 there, that the following inequality holds for $z\ge 0$, 
$$
\pr(S_{k}\in[z,z+1), \tau > k) \le C\frac{1+z}{k^{3/2}}.
$$
We will need an improvement of this inequality, where one has an explicit expression for the constant $C$. 

Let $H(x)$ denote the renewal function corresponding to the process of strict ascending ladder epochs.
\begin{lemma}
\label{lem:concrete.bound}
For every $n\ge4$ one has, uniformly in $x$, 
\begin{align*}
&\pr(S_{n}\in[x,x+1), \tau > n)\\
&\hspace{1cm}\le 
\sqrt{56}\gamma_1(1)
\left(\gamma_1(1)+\gamma_1(x)\right)\frac{\pr(\tau>n/4)}{n}\\
&\hspace{2cm}
+2\sqrt{2}\gamma_1(1)\frac{H(x+1)+2H(1)}{n^{3/2}}.
\end{align*}
\end{lemma}
\begin{proof}
Following \cite{WV09} we denote 
$$
b_k(z):=\pr(S_k\in[z,z+1),\tau>k)
\quad\text{and}\quad 
B_k(z):=\pr(S_k\in(0,z),\tau>k).
$$
According to equation (51) in \cite{WV09},
\begin{align}
\label{eq:cb1}
\nonumber
nb_n(x)&=\pr(S_n\in[x,x+1))
+\sum_{k=1}^{n-1}\int_0^xb_{n-k}(x-y)\pr(S_k\in dy)\\
&\hspace{2cm}
+\sum_{k=1}^{n-1}\int_x^{x+1}B_{n-k}(x+1-y)\pr(S_k\in dy).
\end{align}
It follows from \eqref{sn.tau} that, for all $k\le n-2\lfloor n/4\rfloor$,
\begin{align}
\label{eq:cb1a}
b_{n-k}(x-y)\le\frac{\gamma}{\sqrt{n-k}}\pr(\tau>(n-k)/2)
\le\frac{\gamma}{\sqrt{2\lfloor n/4\rfloor}}\pr(\tau>n/4),
\end{align}
where
$$
\gamma:=\gamma_1(1).
$$
Consequently,
$$
\sum_{k=1}^{n-2\lfloor n/4\rfloor}\int_0^xb_{n-k}(x-y)\pr(S_k\in dy)
\le\frac{\gamma}{\sqrt{2\lfloor n/4\rfloor}}\pr(\tau>n/4)
\sum_{k=1}^{n-2\lfloor n/4\rfloor}\pr(S_k\in[0,x)).
$$
Applying now \eqref{eq.conc}, we have 
\begin{align*}
\sum_{k=1}^{n-2\lfloor n/4\rfloor}\pr(S_k\in[0,x))
 &\le
 \sqrt{2}\gamma_1(x)
\sum_{k=1}^{n}k^{-1/2}\\
 &\le 2^{3/2}\gamma_1(x)\sqrt n. 
 \end{align*}
 As a result, for all $n\ge4$,
 \begin{align}
 \label{eq:cb2}
 \nonumber
 &\sum_{k=1}^{n-2\lfloor n/4\rfloor}\int_0^xb_{n-k}(x-y)\pr(S_k\in dy)\\
 \nonumber
 &\hspace{1cm}
 \le \frac{2\gamma \gamma_1(x)\sqrt{n}}{\sqrt{\lfloor n/4\rfloor}}\pr(\tau>n/4)\\
 &\hspace{1cm}
 \le \sqrt{28}\gamma \gamma_1(x)\pr(\tau>n/4).
 \end{align}
 To estimate the sum over $k>n-2\lfloor n/4\rfloor$, we notice that
 \begin{align*}
 &\int_{j}^{j+1}b_{n-k}(x-y)\pr(S_k\in dy)\\
 &\hspace{1cm}
 =\int_{j}^{j+1}\left(B_{n-k}(x-y+1)-B_{n-k}(x-y)\right)\pr(S_k\in dy)\\
 &\hspace{1cm}
 \le \left(B_{n-k}(x-j+1)-B_{n-k}(x-j-1)\right)\pr(S_k\in [j,j+1]).
 \end{align*}
Applying now \eqref{eq.conc}, we infer that
\begin{align*}
\int_{j}^{j+1}b_{n-k}(x-y)\pr(S_k\in dy)
&\le \frac{\gamma}{\sqrt{k}}\left(B_{n-k}(x-j+1)-B_{n-k}(x-j-1)\right)\\
&\le \frac{\sqrt{2}\gamma}{\sqrt{n}}\left(B_{n-k}(x-j+1)-B_{n-k}(x-j-1)\right)
\end{align*}
for all $k>n-2\lfloor n/4\rfloor$. Consequently,
\begin{align}
\label{eq:cb3}
\nonumber
&\sum_{k=n-2\lfloor n/4\rfloor+1}^{n-1}\int_0^xb_{n-k}(x-y)\pr(S_k\in dy)\\
\nonumber
&\hspace{1cm}\le \frac{\sqrt{2}\gamma}{\sqrt{n}}
\sum_{k=n-2\lfloor n/4\rfloor+1}^{n-1}\sum_{j=0}^{\lfloor x\rfloor}
\left(B_{n-k}(x-j+1)-B_{n-k}(x-j-1)\right)\\
&\hspace{1cm}
\nonumber
\le\frac{\sqrt{2}\gamma}{\sqrt{n}}
\sum_{k=n-2\lfloor n/4\rfloor+1}^{n-1}
\sum_{j=0}^{\lfloor x\rfloor}
\pr(S_{n-k}\in [x-j-1,x-j+1),\tau>n-k)\\
&\hspace{1cm}
\nonumber
\le 2\frac{\sqrt{2}\gamma}{\sqrt{n}}
\sum_{k=n-2\lfloor n/4\rfloor+1}^{n-1}
\pr(S_{n-k}\le x+1,\tau>n-k)\\
&\hspace{1cm}
\le\frac{2\sqrt{2}\gamma}{\sqrt{n}}H(x+1),
\end{align}
here we have used equality 
\[
H(y)=1+\sum_{k=1}^\infty B_k(y). 
\]
Since $z\mapsto B_k(z)$ is increasing for every $k$,
\begin{align*}
\sum_{k=1}^{n-1}\int_x^{x+1}B_{n-k}(x+1-y)\pr(S_k\in dy)
&\le \sum_{k=1}^{n-1}B_{n-k}(1)\pr(S_k\in [x,x+1))\\
&= \sum_{k=1}^{n-1}b_{n-k}(0)\pr(S_k\in [x,x+1)).\\
\end{align*}
Combining \eqref{eq:cb1a} and \eqref{eq.conc}, we have 
\begin{align*}
\sum_{k=1}^{n-2\lfloor n/4\rfloor}b_{n-k}(0)\pr(S_k\in [x,x+1))
&\le \frac{ \gamma^2}{2\sqrt{\lfloor n/4\rfloor}}\pr(\tau>n/4)\sum_{k=1}^nk^{-1/2}\\
&\le\sqrt{7}\gamma^2\pr(\tau>n/4).
\end{align*}
Furthermore, \eqref{eq.conc} implies that
\begin{align*}
\sum_{k=n-2\lfloor n/4\rfloor+1}^{n-1}B_{n-k}(1)\pr(S_k\in [x,x+1))
\le \frac{\sqrt{2}\gamma}{\sqrt{n}}H(1).
\end{align*}
Therefore,
\begin{align}
\label{eq:cb4}
\nonumber
&\sum_{k=1}^{n-1}\int_x^{x+1}B_{n-k}(x-y)\pr(S_k\in dy)\\
&\hspace{12pt}
\le \sqrt{7}\gamma^2\pr(\tau>n/4)
+\frac{\sqrt{2}\gamma}{\sqrt{n}}H(1).
\end{align}
Plugging \eqref{eq:cb2}, \eqref{eq:cb3} and \eqref{eq:cb4} into \eqref{eq:cb1} and applying \eqref{eq.conc} to $\pr(S_n\in[x,x+1)])$,
we get the desired estimate.
\end{proof}
We now simplify a bit the previous lemma.
\begin{corollary}
\label{cor:conc-tau}
There exists an absolute constant $C_0$ such that,
for all $x\ge0$,
$$
\pr(S_n\in[x,x+1),\tau>n)\le 
C_0\frac{\E|S_\tau|}{n^{3/2}}\left(\E|X_1|^3x+(\E|X_1|^3)^2\right).
$$
\end{corollary}
\begin{proof}
If $n\ge 30(\E|X_1|^3)^2$ then we can apply \eqref{eq.tau2} to get 
$$
\pr(\tau>n/4)\le 12\frac{\E|S_\tau|}{\sqrt{n}}.
$$
This bound implies that 
$$
\sqrt{56}\gamma_1(1)
\left(\gamma_1(1)+\gamma_1(x)\right)\frac{\pr(\tau>n/4)}{n}
\le c_1 \frac{\E|S_\tau|}{n^{3/2}}\left(\E|X_1|^3x+(\E|X_1|^3)^2\right)
$$
with some absolute constant $c_1$.
Using once again Mogulskii's bound and Theorem~5 from \cite{Aleshkyavichene06}, we get
$$
H(x)\le \frac{x}{\E\chi^+}+\frac{\E(\xi^+)^2}{\E\chi^+}
\le \frac{x}{\E\chi^+}+c_2\frac{\E|X_1|^3}{\E\chi^+},
$$
where $\chi^+$ is the first weak ascending ladder epoch and $c_2$ is an absolute constant. It is well known that $\E\chi^+\E|S_\tau|=\frac{\E X_1^2}{2}=\frac12$. This implies that 
$$
H(x)\le2\E|S_\tau|\left(x+c_2\E|X_1|^3\right).
$$
Using these estimates, we conclude that 
$$
2\sqrt{2}\gamma_1(1)\frac{H(x+1)+2H(1)}{n^{3/2}}\le 
c_3 \frac{\E|S_\tau|}{n^{3/2}}\left(\E|X_1|^3x+(\E|X_1|^3)^2\right)
$$
with some absolute constant $c_3$. This finishes the proof of the corollary in the case $n\ge 30(\E|X_1|^3)^2$. If $n< 30(\E|X_1|^3)^2$
then the desired inequality is obvious.
\end{proof}
\begin{lemma}
\label{lem:S_tau}
There exists an absolute constant $C_1$ such that 
\begin{align*}
&\pr(\tau=k)
\le C_1\frac{\E|S_\tau|}{k^{3/2}}(\E|X_1|^3)^2,\\
&\E[|S_\tau|;\tau=k]
\le C_1\frac{\E|S_\tau|}{k^{3/2}}(\E|X_1|^3)^2
\end{align*}
and 
$$
\E[\gamma_1(|S_\tau|);\tau=k]
\le C_1\frac{\E|S_\tau|}{k^{3/2}}(\E|X_1|^3)^3.
$$
\end{lemma}
\begin{proof}
All the mentioned in the lemma inequalities are obvious in the case $k=1$.
Assume from now on that $k\ge2$ and consider for some fixed $a,b\ge 0$ and the expectation $\E[(a|S_\tau|+b);\tau=k]$. By the total probability law,
\begin{align*}
\E[(a|S_\tau|+b);\tau=k]
&\le \int_0^\infty \pr(S_{k-1}\in dx,\tau>k-1)\E[(-aX+b);X\le-x]\\
&\le\sum_{j=0}^\infty \pr(S_{k-1}\in [j,j+1),\tau>k-1] \E[(-aX+b);X\le-j]).
\end{align*}
Applying now Corollary~\ref{cor:conc-tau}, we get 
\begin{align*}
&\E[(a|S_\tau|+b);\tau=k]\\
&\hspace{1cm}\le 2^{3/2}C_0\frac{\E|S_\tau|}{k^{3/2}}\E|X_1|^3
\sum_{j=0}^\infty \left(j+\E|X_1|^3\right)\E[(-aX+b);X\le-j])\\
&\hspace{1cm}
=2^{3/2}C_0\frac{\E|S_\tau|}{k^{3/2}}\E|X_1|^3
\E\left[(-aX+b)\sum_{j\in[0,-X]}\left(j+\E|X_1|^3\right);X<0\right]\\
&\hspace{1cm}
\le 2^{3/2}C_0\frac{\E|S_\tau|}{k^{3/2}}\E|X_1|^3
\E\left[(-aX+b)(-X+1)\left(\frac{-X}{2}+\E|X_1|^3\right);X<0\right].
\end{align*}
Taking here $a=0$ and $b=1$, we get 
$$
\pr(\tau=k)\le 2^{3/2}C_0\frac{\E|S_\tau|}{k^{3/2}}\E|X_1|^3
\E\left[(-X+1)\left(\frac{-X}{2}+\E|X_1|^3\right);X<0\right]
$$
Using next the Jensen inequality and noting that $\sigma^2=1$ implies that $\E|X_1|^3\ge1$, we conclude that 
$$
\pr(\tau=k)\le 3\cdot2^{3/2}C_0\frac{\E|S_\tau|}{k^{3/2}}(\E|X_1|^3)^2.
$$
Choosing $a=1$ and $b=0$ and applying once again the Jensen inequality, we infer that the same inequality is valid for $\E[-S_\tau;\tau=k]$.
Finally, taking $a=\E|X_1|^3$ and $b=\pi^{-1/2}$, we obtain the last claim.
\end{proof}
We now can estimate a part of the last sum in \eqref{eq:Pn-def}. 
It follows from~\eqref{eq.conc} that 
\begin{align*}
    &\pr(S_{n-k}+U\in(-x,-x+y]\cup[,x,x+y)])\\
    &\hspace{1cm}= 
    \int_{-\infty}^\infty 
    \pr(U\in du) 
    \pr\left( S_{n-k} +u \in (-x, -x+y] \cup [x, x+y) \right)\\
    &\hspace{1cm}\le \sqrt{2}\frac{\gamma_1(y)}{\sqrt{n-k}}. 
\end{align*}
This implies that
\begin{multline*} 
    \sum_{k=\lf n/2 \rf+1}^{n-1} 
    \int_{0}^\infty
        \pr(\tau = k, S_k\in -dy)
        \pr\left( S_{n-k} +U \in (-x, -x+y] \cup [x, x+y) \right)
\\
\le
    \sqrt{2}
    \sum_{k=\lf n/2 \rf+1}^{n-1}
        \frac{1}{\sqrt{n-k}}
        \E[\gamma_1(|S_\tau|);\tau=k].
\end{multline*}
Applying now the last claim in Lemma~\ref{lem:S_tau}, we obtain
\begin{align}
\label{eq:sum_sesond_half}
\nonumber
&\sum_{k=\lf n/2 \rf+1}^{n-1} 
    \int_{0}^\infty
        \pr(\tau = k, S_k\in -dy)
        \pr\left( S_{n-k} +U \in (-x, -x+y] \cup [x, x+y) \right)\\
\nonumber 
 &\hspace{1cm}\le \sqrt{2}C_1\E|S_\tau|(\E|X_1|^3)^3 
   \sum_{k=\lf n/2 \rf+1}^{n-1} \frac{1}{k^{3/2}(n-k)^{1/2}}\\
\nonumber 
 &\hspace{1cm}\le 4C_1\frac{\E|S_\tau|}{n^{3/2}}(\E|X_1|^3)^3 
   \sum_{k=\lf n/2 \rf+1}^{n-1} \frac{1}{(n-k)^{1/2}}\\
   &\hspace{1cm}\le 8C_1\frac{\E|S_\tau|}{n}(\E|X_1|^3)^3
\end{align}
To obtain an estimate for the sum over $k\le \lf n/2\rf$ we will need an estimate for the density $f_{S_n+U}(z)$ of $U+S_n$. 
\begin{lemma}\label{lem:smooth.dens}
For any $z$, 
\[
|f_{S_n+U}(z) -\frac{1}{\sqrt{2\pi n}} e^{-\frac{z^2}{2n}}|\le 
\left(
\frac{72\E|X_1|^3}{\pi }
+\frac{\E|U|}{\sqrt{2\pi e}}
\right)
\frac{1}{n}
=:C_2\frac{\E|X_1|^3}{n}. 
\]
\end{lemma}
\begin{proof}
Let $Z_n\sim N(0,n)$ and $f_{Z_n+U}(z)$ be the probability density function of $Z_n+U$. We have, by the inversion formula,
\begin{align*}
f_{S_n+U}(z)-
f_{Z_n+U}(z)
&=\frac{1}{2\pi} 
\int_{-\infty}^{\infty}
e^{-itz}
(\phi_{S_n+U}(t)
-\phi_{Z_n+U}(t)
)dt \\ 
&=\frac{1}{2\pi} 
\int_{-2A}^{2A} 
e^{-itz}
\phi_U(t)
(\phi_{S_n}(t)
-\phi_{Z_n}(t)
)dt \\
&=
\frac{1}{2\pi \sqrt{n}}
\int_{-2A\sqrt{n}}^{2A\sqrt{n}}
e^{-itz}
\phi_U(t/\sqrt n)
(\phi_{S_n}(t/\sqrt{n})
-e^{-t^2/2}
),
\end{align*}
where $\phi_{Y}$ denotes the characteristic function of a random variable $Y$.

It was shown in \cite[Lemma V.1]{Petrov} that for $L_n = \E |X_1|^3/\sqrt{n}$ it holds
\begin{align}
\label{eq:phi1}
    \left|
        \phi_{S_n}\left(\frac{t}{{\sqrt{n}}}
        \right)
    \right|
\le
    e^{-t^2/3}, 
\quad
    |t|
\le
    \frac{1}{4L_n}
\end{align}
and
\begin{align}
\label{eq:phi2}
    \left|
        \phi_{S_n}\left(\frac{t}{{\sqrt{n}}}
        \right)
    -
        e^{-t^2/2}
    \right|
\le
    16 L_n
    |t|^3
    e^{-t^2/3},
\quad
    |t| \le \frac{1}{4 L_n}.
\end{align}
Then, 
\begin{align*}
|f_{S_n+U}(z)-
f_{Z_n+U}(z)|
&\le 
\frac{1}{2\pi \sqrt{n}}
\int_{-2A\sqrt{n}}^{2A\sqrt{n}}
|\phi_{S_n}(t/\sqrt{n})
-e^{-t^2/2}
|dt\\
&\le \frac{8\E|X|^3}{ \pi n}
\int_{-2A\sqrt{n}}^{2A\sqrt{n}}
|t|^3
    e^{-t^2/3} dt 
    \le \frac{72\E|X_1|^3}{\pi n}. 
\end{align*}
We next consider the difference 
$
|f_{Z_n+U}(z)-f_{Z_n}(z)|$. 
We first  estimate the difference 
\[
f_{Z_n}(z-u)-f_{Z_n}(z) 
=\frac{1}{\sqrt{2\pi n}}\left(e^{-\frac{(z-u)^2}{2n}}
-
e^{-\frac{-z^2}{2n}}\right). 
\]
For that note that for 
any reals $u,v$ we have, 
by the mean value theorem, 
\begin{equation}\label{eq:mean.value}
\left|
    e^{-\frac{u^{2}}{2n}}
    -e^{-\frac{v^{2}}{2n}}
\right|
=\left|\frac{|u|-|v|}{\sqrt{2n}}\right|
\left|2\theta e^{-\theta^2}\right|
\le \frac{||u|-|v||}{e^{1/2}\sqrt{n}}. 
\end{equation}
Hence, 
\[
|f_{Z_n}(z-u)-f_{Z_n}(z)| 
\le \frac{|u|}{{\sqrt{2\pi e}n}} 
\]
As a result, 
\[
|f_{Z_n+U}(z)-f_{Z_n}(z)|
\le \frac{\E|U|}{\sqrt{2\pi e}n},
\]
which completes the proof. 
\end{proof}
It follows from Lemma~\ref{lem:smooth.dens} that, for every $k\le n/2$,
\begin{multline*}
\Big|\pr\left( S_{n-k} +U \in (-x, -x+y] \cup [x, x+y) \right) \\
-
\frac{1}{\sqrt{2\pi (n-k)}} 
\int_{(-x, -x+y] \cup [x, x+y) }
e^{-\frac{z^2}{2(n-k)}}dz 
\Big|
\le \frac{4C_2\E|X_1|^3}{n}y. 
\end{multline*}
We now show that  
\begin{align}
\label{eq:cont_of_normal}
\left|
    \frac{1}{\sqrt{n-k}} e^{-\frac{z^2}{2(n-k)}}
-  
    \frac{1}{\sqrt{n}} e^{-\frac{z^2}{2n}}
\right|
\le
    \frac{2^{3/2}}{e}\frac{k}{n^{3/2}}
\end{align}
uniformly in $z$ and in $k \le n/2$. To derive this bound we fist notice that
\begin{align*}
&\frac{1}{\sqrt{n-k}} e^{-\frac{z^2}{2(n-k)}}
-  
    \frac{1}{\sqrt{n}} e^{-\frac{z^2}{2n}}\\
&=\left(\frac{1}{\sqrt{n-k}}-\frac{1}{\sqrt{n}}\right)e^{-\frac{z^2}{2n}}  
-
\frac{1}{\sqrt{n-k}}\left(e^{-\frac{z^2}{2n}}
-e^{-\frac{z^2}{2(n-k)}}\right). 
\end{align*}
Since both terms on the right hand side are positive,
\begin{align*}
&\left|\frac{1}{\sqrt{n-k}} e^{-\frac{z^2}{2(n-k)}}
-  
    \frac{1}{\sqrt{n}} e^{-\frac{z^2}{2n}}\right|\\
&\le\max\left\{\left(\frac{1}{\sqrt{n-k}}-\frac{1}{\sqrt{n}}\right), 
\frac{1}{\sqrt{n-k}}\left(e^{-\frac{z^2}{2n}}
-e^{-\frac{z^2}{2(n-k)}}\right)\right\}.
\end{align*}
For the first part of the maximum we have, uniformly in $k\le n/2$,
\begin{align*}
\frac{1}{\sqrt{n-k}}-\frac{1}{\sqrt{n}}
& =\frac{\sqrt{n}-\sqrt{n-k}}{\sqrt{n}\sqrt{n-k}}
  =\frac{k}{\sqrt{n}\sqrt{n-k}(\sqrt{n}+\sqrt{n-k})}\\
& \le \frac{\sqrt{2}k}{(1+1/\sqrt{2})n^{3/2}}\le \frac{k}{n^{3/2}}.
\end{align*}
Using the inequality $1-e^{-u}\le u$ for $u\ge0$, we obtain 
\begin{align*}
\frac{1}{\sqrt{n-k}}\left(e^{-\frac{z^2}{2n}}
-e^{-\frac{z^2}{2(n-k)}}\right)
&\le\frac{\sqrt{2}}{\sqrt{n}} e^{-\frac{z^2}{2n}}
\left(1-e^{-\frac{z^2}{2(n-k)}+\frac{z^2}{2n}}\right)\\
&\le \frac{\sqrt{2}}{\sqrt{n}} e^{-\frac{z^2}{2n}}
\left(\frac{z^2}{2(n-k)}-\frac{z^2}{2n}\right)
=\frac{\sqrt{2}}{\sqrt{n}} e^{-\frac{z^2}{2n}}\frac{z^2}{2n}\frac{k}{n-k}\\
&\le \frac{2\sqrt{2}k}{n^{3/2}}\max_{v\ge0}ve^{-v}
=\frac{2\sqrt{2}k}{en^{3/2}}.
\end{align*}
Noting that $2\sqrt{2}>e$, we arrive at \eqref{eq:cont_of_normal}.
Applying this bound, we get
\begin{multline*}
\left|\pr\left( S_{n-k} +U \in (-x, -x+y] \cup [x, x+y) \right) -
\frac{1}{\sqrt{2\pi n}} 
\int_{(-x, -x+y] \cup [x, x+y) }
e^{-\frac{z^2}{2n}}dz 
\right|\\
\le \frac{4C_2\E|X_1|^3}{n}y+\frac{2^{3/2}}{e}\frac{ky}{n^{3/2}}. 
\end{multline*}

Upper bound~\eqref{eq:mean.value} implies that
\begin{align*}
\left|
   \int_{(-x, -x+y] \cup [x, x+y) }
e^{-\frac{z^2}{2n}}dz  - 2y 
e^{-\frac{x^{2}}{2n}}
\right|
\le
    \frac{2y^2}{e^{1/2}\sqrt{n}}
\end{align*}
and, consequently,
\begin{align*}
&\Bigg|
    \sum_{k=1}^{\lf n/2 \rf}
    \int_{0}^\infty
        \pr(S_k \in -dy, \tau = k)
        \pr\left( S_{n-k} +U \in (-x, -x+y] \cup [x, x+y) \right)\\
&\hspace{4cm}-
    \frac{2}{\sqrt{2\pi n}}
    e^{-\frac{x^2}{2n}}
    \sum_{k=1}^{\lf n/2 \rf}
        \int_0^\infty
        y\pr(S_k \in -dy, \tau = k)
\Bigg|\\
&\hspace{1cm}\le 
    \frac{\sqrt{2}}{\sqrt{e\pi}n} 
    \sum_{k=1}^{\lf n/2 \rf}
    \E[S_\tau^2;\tau=k]
   + \frac{4C_2\E|X_1|^3}{n} \sum_{k=1}^{\lf n/2 \rf} \E[|S_\tau|;\tau=k]   \\ 
   &\hspace{2cm}+
   \frac{2^{3/2}}{en^{3/2}}
   \sum_{k=1}^{\lf n/2 \rf}
    k\E[|S_\tau|;\tau=k].
\end{align*}
We are left to notice that 
\begin{align*}
&\sum_{k=1}^{\lf n/2 \rf}
    \E[S_\tau^2;\tau=k]\le \E S_\tau^2,\\
&\sum_{k=1}^{\lf n/2 \rf}
    \E[|S_\tau|;\tau=k]\le \E |S_\tau|,    
\end{align*}
and, by Lemma~\ref{lem:S_tau},
\begin{align*}
    &\sum_{k=1}^{\lf n/2 \rf}
    k\E[|S_\tau|;\tau=k]
    \le 
    C_1\E|S_\tau|(\E|X_1|^3)^2
    \sum_{k=1}^{\lf n/2 \rf} k^{-1/2}
    \le 2C_1\E|S_\tau|(\E|X_1|^3)^2\sqrt{n}. 
\end{align*}
Hence,
\begin{align*}
 \Bigg|
    \sum_{k=1}^{\lf n/2 \rf}
    \int_{0}^\infty
        \pr(S_k \in -dy, \tau = k)
        &
        \pr\left( S_{n-k} +U \in (-x, -x+y] \cup [x, x+y) \right)\\
&-
    \frac{2}{\sqrt{2\pi n}}
    e^{-\frac{x^2}{2n}}
    \E|S_\tau|
\Bigg|\\
\le   
\frac{C_3(\E|X_1|^3)^2}{n},
\end{align*}
here we have used Lemma~\ref{lem:lorden} to estimate $\E S_\tau$ and $\E S_\tau^2$.
Combining this with \eqref{eq:sum_sesond_half}, we conclude that 
\begin{align*}
&\Bigg|\sum_{k=1}^{n}
    \int_{0}^\infty
        \pr(\tau = k, S_k\in -dy)
        \pr\left( S_{n-k} +U \in (-x, -x+y] \cup [x, x+y) \right)\\
      &\hspace{4cm}-
    \frac{2}{\sqrt{2\pi n}}
    e^{-\frac{x^2}{2n}}
    \E|S_\tau|
\Bigg|\\  
&\hspace{2cm}\le \frac{C_4(\E|X_1|^3)^3}{n},
\end{align*}
with some absolute constant $C_4$.

To estimate other terms in $P_n(x)$ we notice that, due to
the Berry-Esseen inequality~\eqref{eq.berry-esseen.classical}, 
\begin{align*}
    \big|
        \pr(S_{n-k}+U \ge x) 
    -
        \pr(S_{n-k} +U \le -x)
    \big|
\le
    \gamma_0\frac{\E |X_1|^3}{\sqrt{n-k}}.
\end{align*}
Combining this with the estimate
$\pr(\tau=k)
\le C_1\frac{\E|S_\tau|}{k^{3/2}}(\E|X_1|^3)^2$, we get
\begin{align*}
    &\Bigg|
        \sum_{k=\lf n/2 \rf}^{n-1}
            \pr(\tau = k)
            \big[
                \pr(S_{n-k}+U \ge x) 
             -
                \pr(S_{n-k}+U \le -x)
            \big]
    \Bigg|\\
&\hspace{1cm}\le
    2^{3/2}C_1\frac{\E|S_\tau|}{n^{3/2}}(\E|X_1|^3)^2
    \sum_{k=\lf n/2 \rf}^{n-1}
        \frac{1}{\sqrt{n-k}}
\le
    2^{5/2}C_1\frac{\E|S_\tau|}{n}(\E|X_1|^3)^2.
\end{align*}
Thus, we are left to show that
\begin{multline}
    \label{eq:generalreminder.2}
    \Bigg|\pr(S_n+U \ge x)
-
    \pr(S_n+U \le -x)\\
-
    \sum_{k=1}^{\lf n/2 \rf}
        \pr(\tau = k)
        \big[
            \pr(S_{n-k} +U\ge x) 
         -
            \pr(S_{n-k} +U\le -x)
        \big]\Bigg|
\le \frac{C_4\E|S_\tau|(\E|X_1|^3)^3}{n}
\end{multline} 
uniformly in $x$, with some absolute constant $C_4$.
Write
\begin{align*}
    1 
=
    \pr(\tau > \lf n/2 \rf)
+    
    \sum_{k=1}^{\lf n/2 \rf} \pr(\tau = k).
\end{align*}
Using once again the classical Berry-Esseen inequality, we have
$$
\big|\pr(S_n \ge x) - \pr(S_n \le -x)\big| \le 2\gamma_0\E|X_1|^3/\sqrt{n}.
$$
Furthermore, Lemma~\ref{lem:S_tau} implies that
$$
\pr(\tau > \lf n/2 \rf) 
\le 
C_1\frac{\E|S_\tau|(\E|X_1|^3)^2}{(\lf n/2 \rf)^{1/2}}.
$$ 
Thus,
\eqref{eq:generalreminder.2} is equivalent to
\begin{align*}
    \left|\sum_{k=1}^{\lf n/2 \rf}
        \pr(\tau = k)
        \big[
        Q_{n-k}(x) - Q_{n}(x)
        \big]\right|
\le 
    \frac{C_5\E|S_\tau|(\E|X_1|^3)^3}{n},
\end{align*}
where
\begin{align*}
    Q_{n}(x) 
=
    \pr(S_n+U \ge x) - \pr(S_n+U \le -x).
\end{align*}
This estimate is an easy consequence of the following lemma.
\begin{lemma}
    For $k\le n/2$ it holds, uniformly in $x$,
\begin{align*}
    \big|Q_{n-k}(x) - Q_{n}(x)\big| 
\le
    109\sqrt{\frac{3}{\pi}}2^{3/2}\E|X_1|^3\frac{k}{n^{3/2}}.
\end{align*}
\end{lemma}
\begin{proof}
Obviously it is enough to bound the distance between neighbouring functions $Q_{n+1}(x)$ and $Q_{n}(x)$. 
Let $\ph(t)$ denote the characteristic function of $X$. 
Applying the inversion formula, we have
\begin{align*}
f_{S_n+U}(x)-f_{S_n+U}(-x)
=
    \frac{1}{2\pi}
    \int_{-2A}^{2A}
        e^{-itx}
        \phi_U(t)
        \big[
            \ph^n(t)
        -
            \ph^n(-t)
        \big]
    dt
\end{align*}
and
\begin{align*}
    Q_n(x)
=&
\pr(S_{n}+U>x)-\pr(S_{n}+U<-x)
    \\
=&
\lim_{T\to \infty} 
\int_x^T (f_{S_n+U}(y)-f_{S_n+U}(-y))dy \\
&=\frac{1}{2\pi}
\lim_{T\to \infty} 
    \int\limits_{-2A}^{2A}
        \frac{e^{-itx}-e^{-itT}}{it}
        \phi_U(t)
        \big[
            \ph^n(t)
        -
            \ph^n(-t)
        \big]
    dt.
\end{align*}
Consequently, uniformly in $x$,
\begin{align*}
    \big|
        Q_{n+1}(x)
    -
        Q_n(x)
    \big|
\le
    \frac{1}{\pi}
    \int\limits_{-2A}^{2A}
        \frac{
        \big|
            \ph^{n}(t) 
        -
            \ph^{n}(-t)
        -
            \ph^{n+1}(t)
        +
            \ph^{n+1}(-t)
        \big|
        }{|t|}
    dt.
\end{align*}
To estimate the nominator we first notice that
\begin{align*}
    \ph^{n}(t) - \ph^{n}(-t)
-
    \ph^{n+1}(t) + \ph^{n+1}(-t)
=&
    \big[\ph^n(t) - \ph^n(-t)\big](1-\ph(t))\\
+&
    \big[\ph(-t) - \ph(t)\big]\ph^n(-t) .
\end{align*}
This implies that
\begin{align*}
    \big|
        \ph^{n}(t) - \ph^{n}(-t)
    -
        \ph^{n+1}(t) + \ph^{n+1}(-t)
    \big|
\le&
    |\ph^n(t) - \ph^n(-t)|\cdot|1-\ph(t)|\\
+&
    |\ph(-t) - \ph(t)|\cdot|\ph^n(-t)|.
\end{align*}
Using the Taylor formula and recalling that
$\E X = 0$, $\E X^2=1$ and $ \E |X|^3 < \infty$,
we obtain
\begin{align*}
    |1-\ph(t)| \le \frac{\E X_1^2}{2}t^2=\frac{t^2}{2}
\quad
\text{and}
\quad
    |\ph(-t) - \ph(t)|
\le
    \frac{\E|X_1|^3}{3}|t|^3.
\end{align*}
It follows from \eqref{eq:phi1} and \eqref{eq:phi2} that
\begin{align*}
    |\ph^n(t)|
\le
    e^{-t^2n/3}
\quad
\text{and}
\quad
    |\ph^n(t) - e^{-t^2n/2}|\
\le
    16\E|X_1|^3 n|t|^3e^{-t^2n/3}
\end{align*}
for $|t| \le 2A$.
Hence, for such $t$ we have
\begin{align*}
    \big|
        \ph^{n}(t) - \ph^{n}(-t)
    -
        \ph^{n+1}(t) + \ph^{n+1}(-t)
    \big|
\le
    16\E|X_1|^3 n|t|^5 e^{-t^2n/3}
+
    \frac{\E|X_1|^3}{3}|t|^3 e^{-t^2n/3}.
\end{align*}
Consequently,
\begin{align*}
    \big|Q_{n+1}(x)-Q_n(x)\big|
&\le
    \frac{16\E|X_1|^3 n}{\pi}
    \int_{-\infty}^\infty 
        |t|^4e^{-t^2n/3}dt
+
    \frac{\E|X_1|^3}{3\pi}
    \int_{-\infty}^\infty 
        |t|^2e^{-t^2n/3}dt\\
&=
    \frac{\E|X_1|^3}{\pi n^{3/2}}
    \left(
        32\int_0^\infty 
            z^4e^{-z^2/3}dz
    +
        \frac{2}{3}
        \int_0^\infty 
            z^2e^{-z^2/3}dz
    \right)\\
&\le 
    109\sqrt{\frac{3}{\pi}}\frac{\E|X_1|^3}{n^{3/2}}.
\end{align*}
As a result we have
\begin{align*}
    |Q_{n+1}(x) - Q_n(x)|
\le
    109\sqrt{\frac{3}{\pi}}\frac{\E|X_1|^3}{n^{3/2}}.
\end{align*}
Thus, the proof is complete.
\end{proof}
Combining this lemma with the upper bound for $\pr(\tau=k)$ obtained in Lemma~\ref{lem:S_tau}, we have  
\begin{align*}
    \left|
        \sum_{k=1}^{\lf n/2 \rf}
            \pr(\tau = k)
            \big[
            Q_{n-k}(x) - Q_{n}(x)
            \big]
    \right|
&\le
    \frac{109\sqrt{3}\E|X_1|^3}{v{\sqrt{\pi}}n^{3/2}}
    \sum_{k=1}^{\lf n/2 \rf}
        k\pr(\tau = k) \\
\le \frac{C_5\E|S_\tau|(\E|X_1|^3)^3}{n}.
\end{align*}
As a result we have, uniformly in $x\ge0$, 
$$
\left|\pr(S_n\ge x,\tau>n)-\sqrt{\frac{2}{\pi}}\E|S_\tau|n^{-1/2}e^{-x^2/2n}\right|
\le C_7(\E|X_1|^3)^3\frac{\E|S_\tau|}{n}
$$
Taking here $x=0$ we have 
$$
\left|\pr(\tau>n)-\sqrt{\frac{2}{\pi}}\E|S_\tau|n^{-1/2}\right|
\le C_7(\E|X_1|^3)^3\frac{\E|S_\tau|}{n}.
$$
Combining these two inequalities, we get the desired estimate for the conditional distribution.
\bibliographystyle{abbrv}
\bibliography{references}
\end{document}